%% file: main.tex
\documentclass[10pt]{article}
\usepackage{preamble}

\title{Astrolabe: A Content-Addressable Hypergraph\\for Semantic Knowledge Management}
\author{
  Xinze Li\thanks{Department of Mathematics, University of Toronto. \texttt{lixinze@math.toronto.edu}}
}
\date{\today}

\begin{document}

\maketitle

\begin{abstract}
Existing knowledge management tools either preserve prose but lose structural relationships, or capture relationships but restrict edge semantics to fixed vocabularies. We introduce Astrolabe, a content-addressable hypergraph for semantic knowledge management. Entries are identified by the SHA-256 hash of their content, carry an ordered reference list of arbitrary width, and store an opaque record string interpreted by plugins. The structure admits two orthogonal decompositions: by width and by depth. We demonstrate the framework with a plugin bridging informal and formal mathematics.
\end{abstract}

{\footnotesize\tableofcontents}
\newpage
\input{sections/introduction}
\input{sections/data-model}
\input{sections/plugins}
\input{sections/conclusion}

\section*{Acknowledgements}
\addcontentsline{toc}{section}{Acknowledgements}

The author thanks Jiaqi Lai and Alejandro Radisic for exploring the \href{https://github.com/Xinze-Li-Moqian/Astrolabe}{early prototype} together, without which the data structure presented here would not have taken shape. The author is grateful to Alex Kontorovich, Leonardo de Moura, Yevgeny Liokumovich, Simone Severini, and Patrick Shafto for their encouragement and support, to Samuel Schlesinger and Zhifei Zhu for helpful discussions, and to the Lean community for valuable feedback.

\addcontentsline{toc}{section}{References}
\bibliographystyle{alpha}
\bibliography{references}

\end{document}

%% file: sections/introduction.tex
\section{Introduction}
\label{sec:introduction}

Knowledge management tools broadly fall into two categories: document-centric systems (wikis, note-taking apps) that store prose but lose structural relationships, and graph databases that capture relationships but require schema design and lack accessible authoring interfaces. A recurring challenge across both categories is representing \textbf{semantic} relationships. In a mathematical knowledge base, for instance, knowing \textbf{that} Theorem~A depends on Definition~B is less informative than knowing \textbf{how}: does A unfold B, rewrite by B, or apply B to a subterm? Tools such as \texttt{leanblueprint}~\cite{Massot2020} have made significant progress by introducing dependency graphs for formalization projects; Astrolabe builds on this foundation by adding semantic content to each dependency edge.

In this paper, we present Astrolabe, a content-addressable data structure designed for semantic knowledge management. In \S\ref{sec:data-model}, we define the core data structure and two orthogonal decompositions: by width (how many references an entry has) and by depth (how deep the reference chain goes). In \S\ref{sec:plugins}, we present LeanNets, the first plugin built on this framework, which bridges informal (\LaTeX) and formal (Lean~4) mathematics by extracting a directed semantic network from the store.
\subsection{Related Work}

Note-taking tools such as Obsidian~\cite{Obsidian}, Roam Research~\cite{RoamResearch}, and Logseq~\cite{Logseq} introduced bidirectional links, turning notes into graphs. These links carry no semantic content: \texttt{[[page name]]} records a connection but not its nature. Knowledge graphs in the RDF/OWL tradition provide typed edges, but the type vocabulary is a fixed enumeration; adding a new relationship type requires schema modification. Content-addressable systems such as Git and IPFS/IPLD~\cite{IPLD} provide immutable, hash-identified objects, but impose fixed object types and do not support higher-dimensional semantic structures. These systems use \textbf{Merkle DAGs}: each object's hash is computed over both its content and the hashes of its children, so modifying any leaf cascades hash changes to every ancestor. A single root hash therefore authenticates an entire subgraph---the security foundation of Git and IPFS. The same mechanism makes reference cycles unrepresentable, since the hash of a cycle cannot be computed; it also means that editing any object invalidates the identity of every ancestor, and that independently edited replicas diverge in identity even when their content agrees. Astrolabe computes $\mathrm{id}(e) = H(\mathrm{rec}(e))$ with the reference list excluded from the hash. An entry's identity therefore depends only on its record string, not on what it references or what references it: reference cycles are permitted, and editing one entry does not change the identity of another. The trade-off is that hash cascading---and the tamper evidence it provides---is absent from the core. For a knowledge base under active editing, this is acceptable; for applications that do require cascading verification, a plugin can encode reference information into the record field, recovering Merkle-style hash dependencies. Semantic propagation of changes is handled explicitly at the application layer (\S\ref{sec:plugins}).

On the theoretical side, Spivak and Kent~\cite{Spivak2012} proposed \textbf{ologs} (ontology logs) as a categorical framework for knowledge representation (Figure~\ref{fig:olog}). Objects are types, morphisms are functional relations, and the category-theoretic structure enforces consistency. However, ologs restrict morphisms to functional relations and cannot carry open-ended content. The framework has seen limited software adoption.

\begin{figure}[ht]
\centering
\begin{tikzcd}[column sep=large, row sep=large]
\text{a bounded sequence} \arrow[r, "\text{has}"] \arrow[d, "\text{is}"'] & \text{a convergent subsequence} \arrow[d, "\text{is}"] \\
\text{a sequence in } \mathbb{R}^n \arrow[r, "\text{lives in}"'] & \text{a subset of } \mathbb{R}^n
\end{tikzcd}
\caption{An olog~\cite{Spivak2012}: objects are types, morphisms are functional relations, and the diagram commutes.}
\label{fig:olog}
\end{figure}
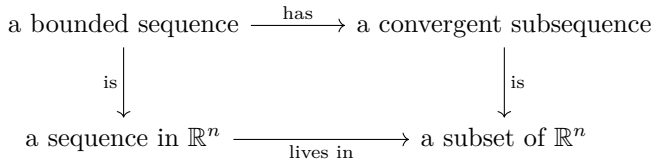

On the engineering side, HyperGraphDB~\cite{HyperGraphDB} is a generalized graph database where every entity is an \textbf{atom} with a target set (a list of references to other atoms). Atoms with an empty target set are pure nodes; atoms with a non-empty target set are hyperedges. Since hyperedges are themselves atoms, they can be referenced by other hyperedges, yielding higher-order relationships (Figure~\ref{fig:hgdb}).

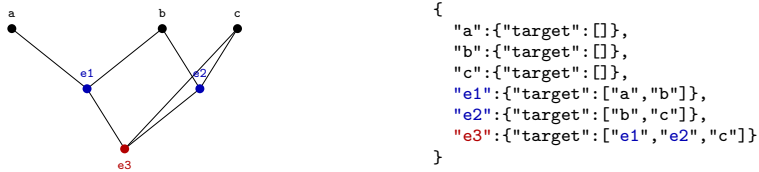
\begin{figure}[ht]
\centering
\begin{minipage}[t]{0.48\textwidth}
\vspace{0pt}
\centering
\begin{tikzpicture}[v/.style={circle, fill, inner sep=0.5pt}]
  \tikzset{s0/.style={circle, fill=black, inner sep=1.2pt}}
  \tikzset{s1/.style={circle, fill=blue!70!black, inner sep=1.2pt}}
  \tikzset{s2/.style={circle, fill=red!70!black, inner sep=1.2pt}}
  \draw[very thin] (1.0,1.0) -- (0.0,1.8);
  \draw[very thin] (1.0,1.0) -- (2.0,1.8);
  \draw[very thin] (2.5,1.0) -- (2.0,1.8);
  \draw[very thin] (2.5,1.0) -- (3.0,1.8);
  \draw[very thin] (1.5,0.2) -- (1.0,1.0);
  \draw[very thin] (1.5,0.2) -- (2.5,1.0);
  \draw[very thin] (1.5,0.2) -- (3.0,1.8);
  \node[s0] (a) at (0.0,1.8) {};
  \node[font=\tiny, above=1pt, text=black] at (a) {\texttt{a}};
  \node[s0] (b) at (2.0,1.8) {};
  \node[font=\tiny, above=1pt, text=black] at (b) {\texttt{b}};
  \node[s0] (c) at (3.0,1.8) {};
  \node[font=\tiny, above=1pt, text=black] at (c) {\texttt{c}};
  \node[s1] (e1) at (1.0,1.0) {};
  \node[font=\tiny, above=1pt, text=blue!70!black] at (e1) {\texttt{e1}};
  \node[s1] (e2) at (2.5,1.0) {};
  \node[font=\tiny, above=1pt, text=blue!70!black] at (e2) {\texttt{e2}};
  \node[s2] (e3) at (1.5,0.2) {};
  \node[font=\tiny, below=1pt, text=red!70!black] at (e3) {\texttt{e3}};
\end{tikzpicture}
\end{minipage}%
\hspace{4pt}%
\begin{minipage}[t]{0.48\textwidth}
\vspace{0pt}
{\scriptsize\ttfamily
\begin{tabular}{@{}l@{}}
\{\\
\ \ \textcolor{black}{"a"}:\{"target":[]\},\\
\ \ \textcolor{black}{"b"}:\{"target":[]\},\\
\ \ \textcolor{black}{"c"}:\{"target":[]\},\\
\ \ \textcolor{blue!70!black}{"e1"}:\{"target":["\textcolor{black}{a}","\textcolor{black}{b}"]\},\\
\ \ \textcolor{blue!70!black}{"e2"}:\{"target":["\textcolor{black}{b}","\textcolor{black}{c}"]\},\\
\ \ \textcolor{red!70!black}{"e3"}:\{"target":["\textcolor{blue!70!black}{e1}","\textcolor{blue!70!black}{e2}","\textcolor{black}{c}"]\}\\
\}
\end{tabular}
}
\end{minipage}
\caption{HyperGraphDB data model: every entity is an atom with a target set. \textcolor{black}{Black} = atoms (target = []), \textcolor{blue!70!black}{blue} = hyperedges, \textcolor{red!70!black}{red} = higher-order hyperedge ($e_3$ references $e_1$ and $e_2$).}
\label{fig:hgdb}
\end{figure}

\subsection{Overview}

Astrolabe adopts the same core abstraction as HyperGraphDB (every entry has an ordered reference list pointing to other entries) and adds content-addressable identity: entries are identified by the SHA-256 hash of their content (as in Git/IPFS~\cite{IPLD}), providing deduplication, per-entry tamper evidence, and decentralized identity. Because identity depends only on content, changes do not propagate through hashes; Astrolabe therefore provides a semantic propagation mechanism that traverses the skeleton graph to identify all entries semantically affected by a change.

The data structure is \textbf{maximally relaxed}. The record field is an opaque string; all domain-specific conventions are delegated to plugins, keeping the core format domain-agnostic. Astrolabe is open-source and available as a Tauri desktop application.\footnote{\url{https://github.com/MathNetwork/Astrolabe}}

%% file: sections/data-model.tex
\section{Data Model}
\label{sec:data-model}

We fix a hash function $H : \Sigma^* \to \Sigma^*$ throughout. The core data structure of Astrolabe is a content-addressable store of \textbf{nerves}, where each nerve is a triple of identity, ref, and record. Nerves can have arbitrary-length references, forming higher-dimensional semantic structures. The store admits two orthogonal decompositions: by \textbf{width} (how many references) and by \textbf{depth} (the depth of the reference chain). We describe each in turn.

\subsection{Content-addressable Hypergraph}

\begin{definition}[AstroNerve]\label{def:nerve}
An \textbf{AstroNerve} is a triple $(\mathrm{id},\, \mathrm{ref},\, \mathrm{rec})$ where $\mathrm{id} \in \Sigma^*$ is the identity, $\mathrm{ref} = (r_1, \ldots, r_k)$ is an ordered list of references ($r_i \in \Sigma^*$), and $\mathrm{rec} \in \Sigma^*$ is the record string.  The \textbf{width} of an AstroNerve is $\mathrm{w}(e) = |\mathrm{ref}(e)| - 1$.
\end{definition}

\begin{definition}[AstroNet]\label{def:ca-hypergraph}\label{def:well-formed}
Let $H : \Sigma^* \to \Sigma^*$ be a hash function.  An \textbf{AstroNet} (or \textbf{well-formed content-addressable hypergraph}) is a finite set $A$ of nerves together with $H$, satisfying:
\begin{enumerate}
  \setcounter{enumi}{-1}
  \item \textbf{(Content-addressing)} For every $e \in A$, $\mathrm{id}(e) = H(\mathrm{rec}(e))$.
  \item \textbf{(Self-reference)} Every nerve $e$ with $\mathrm{w}(e) = 0$ satisfies $\mathrm{ref}(e) = (\mathrm{id}(e))$.
  \item \textbf{(Injectivity)} The map $e \mapsto \mathrm{id}(e)$ is injective on $A$.
  \item \textbf{(Closure)} For every $e \in A$ and $r_i \in \mathrm{ref}(e)$, there exists $e' \in A$ with $\mathrm{id}(e') = r_i$.
  \item \textbf{(No duplicates)} For every nerve $e$ with $\mathrm{w}(e) > 0$, the elements of $\mathrm{ref}(e)$ are pairwise distinct.
  \item \textbf{(No self-reference)} For every nerve $e$ with $\mathrm{w}(e) > 0$, $\mathrm{id}(e) \notin \mathrm{ref}(e)$.
\end{enumerate}
\end{definition}

A nerve $e \in A$ with $\mathrm{w}(e) = 0$ is called an \textbf{atom}; by the self-reference axiom, $\mathrm{ref}(e) = (\mathrm{id}(e))$.  Non-emptiness of $\mathrm{ref}(e)$ follows: atoms have length~1, and nerves with $\mathrm{w}(e) > 0$ have $|\mathrm{ref}(e)| \geq 2$ by the definition of width.  All networks in the sequel are AstroNets.

The format is general-purpose and can serve any knowledge management scenario by defining appropriate record conventions. In the current implementation, the store is serialized as a single JSON file (\texttt{astrolabe.json}):

\medskip
{\small\ttfamily
\begin{tabular}{@{}l@{}}
\{\\
\ \ \textcolor{gray}{"<12-char-hash>"}:\{\textcolor{blue!70!black}{"ref"}:[\textcolor{gray}{"<hash>"}, ...], \textcolor{red!70!black}{"record"}:\textcolor{gray}{"<string>"}\}\\
\}
\end{tabular}
}
\medskip

\noindent Concretely, each nerve has three parts:
\begin{itemize}
  \item \textbf{ref} --- an ordered list of hashes, any length. $|\texttt{ref}|$ defines the \textbf{width}:
    \begin{itemize}
      \item width 0: \texttt{ref = [self\_hash]} --- an atom (base unit)
      \item width 1: \texttt{ref = [A, B]} --- a binary relation
      \item width $k$: \texttt{ref = [h$_0$, h$_1$, ..., h$_k$]} --- a higher-dimensional semantic relation
    \end{itemize}
  \item \textbf{record} --- a plain string. The core layer does not interpret it; plugins define conventions for structured content (JSON with \texttt{sort}, \texttt{source}, \texttt{title}, \texttt{notes}, etc.).
  \item \textbf{hash} --- $\texttt{SHA256}(\mathit{record})\texttt{[:12 hex]}$, content-addressable.
\end{itemize}

\subsection{Width Decomposition}

Figure~\ref{fig:astrolabe-network} shows a 15-nerve AstroNet and its reference network. The nerves span multiple widths: 4~atoms (\texttt{a1}--\texttt{a4}), 4~width-1 edges (\texttt{e1}--\texttt{e4}), and higher-width nerves that freely mix atoms and edges in their \texttt{ref} lists. The nerves \texttt{c1}, \texttt{c2}, \texttt{c3} form a reference cycle. The nerve \texttt{m2} references \texttt{m1} (which itself references \texttt{f1}, a width-2 nerve), demonstrating that references are unconstrained: any nerve can reference any other nerve, regardless of width.

\begin{figure}[ht]
\centering
\begin{minipage}[t]{0.48\textwidth}
\vspace{0pt}
\centering
\scalebox{1.5}{\input{figures/small_network_width.tex}}
\end{minipage}%
\hspace{4pt}%
\begin{minipage}[t]{0.48\textwidth}
\vspace{0pt}
{\scriptsize\ttfamily
\begin{tabular}{@{}l@{}}
\{\\
\ \ \textcolor{black}{"a1"}:\{"ref":["a1"]\},\\
\ \ \textcolor{black}{"a2"}:\{"ref":["a2"]\},\\
\ \ \textcolor{black}{"a3"}:\{"ref":["a3"]\},\\
\ \ \textcolor{black}{"a4"}:\{"ref":["a4"]\},\\
\ \ \textcolor{blue!70!black}{"e1"}:\{"ref":["a1","a2"]\},\\
\ \ \textcolor{blue!70!black}{"e2"}:\{"ref":["a2","a3"]\},\\
\ \ \textcolor{blue!70!black}{"e3"}:\{"ref":["a3","a4"]\},\\
\ \ \textcolor{blue!70!black}{"e4"}:\{"ref":["a4","a1"]\},\\
\ \ \textcolor{red!70!black}{"f1"}:\{"ref":["a1","e1","e2"]\},\\
\ \ \textcolor{blue!70!black}{"f2"}:\{"ref":["e3","e4"]\},\\
\ \ \textcolor{blue!70!black}{"m1"}:\{"ref":["f1","e3"]\},\\
\ \ \textcolor{red!70!black}{"m2"}:\{"ref":["m1","f2","a2"]\},\\
\ \ \textcolor{blue!70!black}{"c1"}:\{"ref":["c2","a1"]\},\\
\ \ \textcolor{blue!70!black}{"c2"}:\{"ref":["c3","a2"]\},\\
\ \ \textcolor{blue!70!black}{"c3"}:\{"ref":["c1","a3"]\}\\
\}
\end{tabular}
}
\end{minipage}
\caption{Reference network (left) and AstroNet data (right), colored by width: \textcolor{black}{black} = width~0 (atoms), \textcolor{blue!70!black}{blue} = width~1, \textcolor{red!70!black}{red} = width~2. Nerves \texttt{f1} and \texttt{m2} mix atoms and edges in their \texttt{ref}.}
\label{fig:astrolabe-network}
\end{figure}
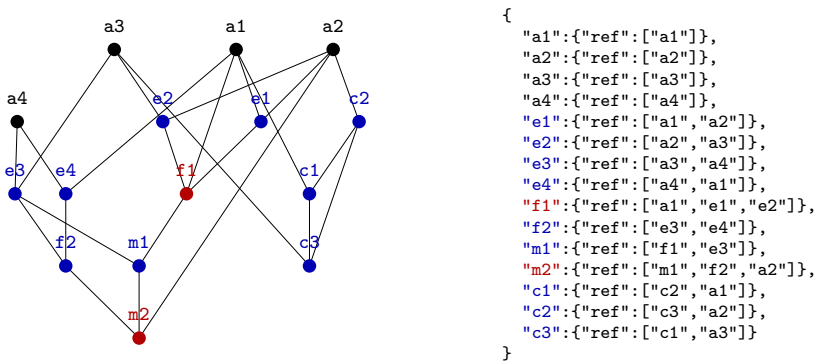

\subsection{Depth Decomposition}

Width counts the number of references but does not distinguish whether those references point to atoms or to deeply nested nerves.  The depth filtration captures this layering.

\begin{definition}[Depth filtration]\label{def:stage}
Given an AstroNet $A$, define $A^{(0)} \subseteq A^{(1)} \subseteq \cdots \subseteq A$ inductively:
\begin{itemize}
  \item $A^{(0)} = \{ e \in A \mid \mathrm{w}(e) = 0 \}$ \quad (atoms),
  \item $A^{(m+1)} = \{ e \in A \mid \forall\, r \in \mathrm{ref}(e),\; r \in A^{(m)} \}$.
\end{itemize}
The \textbf{depth} of $e$ is $\mathrm{depth}(e) = \min\{m \mid e \in A^{(m)}\}$. Nerves involved in reference cycles have no finite depth and are assigned $\mathrm{depth}(e) = -1$.
\end{definition}

\noindent Figure~\ref{fig:depth-coloring} shows the same network colored by depth instead of width. The key contrast: \textcolor{blue!70!black}{\texttt{e1}} and \textcolor{red!70!black}{\texttt{f2}} both have width~1, but \textcolor{blue!70!black}{\texttt{e1}} is depth~1 (refs are atoms) while \textcolor{red!70!black}{\texttt{f2}} is depth~2 (refs include depth-1 nerves).

\begin{figure}[ht]
\centering
\begin{minipage}[t]{0.48\textwidth}
\vspace{0pt}
\centering
\scalebox{1.5}{\input{figures/small_network_depth.tex}}
\end{minipage}%
\hspace{4pt}%
\begin{minipage}[t]{0.48\textwidth}
\vspace{0pt}
{\scriptsize\ttfamily
\begin{tabular}{@{}l@{}}
\{\\
\ \ \textcolor{black}{"a1"}:\{"ref":["a1"]\},\\
\ \ \textcolor{black}{"a2"}:\{"ref":["a2"]\},\\
\ \ \textcolor{black}{"a3"}:\{"ref":["a3"]\},\\
\ \ \textcolor{black}{"a4"}:\{"ref":["a4"]\},\\
\ \ \textcolor{blue!70!black}{"e1"}:\{"ref":["a1","a2"]\},\\
\ \ \textcolor{blue!70!black}{"e2"}:\{"ref":["a2","a3"]\},\\
\ \ \textcolor{blue!70!black}{"e3"}:\{"ref":["a3","a4"]\},\\
\ \ \textcolor{blue!70!black}{"e4"}:\{"ref":["a4","a1"]\},\\
\ \ \textcolor{red!70!black}{"f1"}:\{"ref":["a1","e1","e2"]\},\\
\ \ \textcolor{red!70!black}{"f2"}:\{"ref":["e3","e4"]\},\\
\ \ \textcolor{violet!70!black}{"m1"}:\{"ref":["f1","e3"]\},\\
\ \ \textcolor{green!50!black}{"m2"}:\{"ref":["m1","f2","a2"]\},\\
\ \ \textcolor{gray}{"c1"}:\{"ref":["c2","a1"]\},\\
\ \ \textcolor{gray}{"c2"}:\{"ref":["c3","a2"]\},\\
\ \ \textcolor{gray}{"c3"}:\{"ref":["c1","a3"]\}\\
\}
\end{tabular}
}
\end{minipage}
\caption{Depth coloring of the same network: \textcolor{black}{black} = depth~0, \textcolor{blue!70!black}{blue} = depth~1, \textcolor{red!70!black}{red} = depth~2, \textcolor{violet!70!black}{purple} = depth~3, \textcolor{green!50!black}{green} = depth~4, \textcolor{gray}{gray} = cycle. Note: \texttt{e1} and \texttt{f2} both have width~1, but \texttt{e1} is depth~1 while \texttt{f2} is depth~2.}
\label{fig:depth-coloring}
\end{figure}

\begin{proposition}\label{prop:stage-stabilize}
For any finite AstroNet $A$, the depth filtration stabilizes: there exists $N$ such that $A^{(N)} = A^{(N+1)} = \cdots$. Let $U = A \setminus \bigcup_{m \geq 0} A^{(m)}$ denote the set of nerves with no finite depth. Then: (a)~every nerve on a reference cycle belongs to $U$; (b)~every nerve in $U$ reaches, via a finite chain of references, a nerve on a reference cycle. In particular, $U = \emptyset$ if and only if $A$ is acyclic.
\end{proposition}

\begin{proof}
Since $A$ is finite and $A^{(0)} \subseteq A^{(1)} \subseteq \cdots \subseteq A$, the chain stabilizes.  For~(a), if $e$ lies on a cycle $e_1, \ldots, e_n$, finite depths would give a strictly decreasing cycle $d_1 > d_2 > \cdots > d_n > d_1$, a contradiction.  For~(b), each $e \in U$ has a reference in $U$ (by Axiom~5, distinct from $e$ itself); iterating and applying pigeonhole yields a cycle.
\end{proof}

Since identity is computed from $\mathrm{rec}(e)$ alone, reference cycles do not affect hash computation. Cyclic nerves are assigned $\mathrm{depth}(e) = -1$ for structural analysis purposes.

%% file: figures/small_network_width.tex
\begin{tikzpicture}[v/.style={circle, fill, inner sep=0.5pt}]
  \tikzset{s0/.style={circle, fill=black, inner sep=1.2pt}}
  \tikzset{s1/.style={circle, fill=blue!70!black, inner sep=1.2pt}}
  \tikzset{s2/.style={circle, fill=red!70!black, inner sep=1.2pt}}
  \draw[very thin] (2.43,2.04) -- (2.21,2.68);
  \draw[very thin] (2.43,2.04) -- (3.07,2.68);
  \draw[very thin] (1.56,2.04) -- (3.07,2.68);
  \draw[very thin] (1.56,2.04) -- (1.13,2.68);
  \draw[very thin] (0.25,1.40) -- (1.13,2.68);
  \draw[very thin] (0.25,1.40) -- (0.27,2.04);
  \draw[very thin] (0.70,1.40) -- (0.27,2.04);
  \draw[very thin] (0.70,1.40) -- (2.21,2.68);
  \draw[very thin] (1.77,1.40) -- (2.21,2.68);
  \draw[very thin] (1.77,1.40) -- (2.43,2.04);
  \draw[very thin] (1.77,1.40) -- (1.56,2.04);
  \draw[very thin] (0.70,0.76) -- (0.25,1.40);
  \draw[very thin] (0.70,0.76) -- (0.70,1.40);
  \draw[very thin] (1.35,0.76) -- (1.77,1.40);
  \draw[very thin] (1.35,0.76) -- (0.25,1.40);
  \draw[very thin] (1.35,0.12) -- (1.35,0.76);
  \draw[very thin] (1.35,0.12) -- (0.70,0.76);
  \draw[very thin] (1.35,0.12) -- (3.07,2.68);
  \draw[very thin] (2.86,1.40) -- (3.30,2.04);
  \draw[very thin] (2.86,1.40) -- (2.21,2.68);
  \draw[very thin] (3.30,2.04) -- (2.86,0.76);
  \draw[very thin] (3.30,2.04) -- (3.07,2.68);
  \draw[very thin] (2.86,0.76) -- (2.86,1.40);
  \draw[very thin] (2.86,0.76) -- (1.13,2.68);
  \node[s0] (a1) at (2.21,2.68) {};
  \node[font=\tiny, above=1pt, text=black] at (a1) {\texttt{a1}};
  \node[s0] (a2) at (3.07,2.68) {};
  \node[font=\tiny, above=1pt, text=black] at (a2) {\texttt{a2}};
  \node[s0] (a3) at (1.13,2.68) {};
  \node[font=\tiny, above=1pt, text=black] at (a3) {\texttt{a3}};
  \node[s0] (a4) at (0.27,2.04) {};
  \node[font=\tiny, above=1pt, text=black] at (a4) {\texttt{a4}};
  \node[s1] (e1) at (2.43,2.04) {};
  \node[font=\tiny, above=1pt, text=blue!70!black] at (e1) {\texttt{e1}};
  \node[s1] (e2) at (1.56,2.04) {};
  \node[font=\tiny, above=1pt, text=blue!70!black] at (e2) {\texttt{e2}};
  \node[s1] (e3) at (0.25,1.40) {};
  \node[font=\tiny, above=1pt, text=blue!70!black] at (e3) {\texttt{e3}};
  \node[s1] (e4) at (0.70,1.40) {};
  \node[font=\tiny, above=1pt, text=blue!70!black] at (e4) {\texttt{e4}};
  \node[s2] (f1) at (1.77,1.40) {};
  \node[font=\tiny, above=1pt, text=red!70!black] at (f1) {\texttt{f1}};
  \node[s1] (f2) at (0.70,0.76) {};
  \node[font=\tiny, above=1pt, text=blue!70!black] at (f2) {\texttt{f2}};
  \node[s1] (m1) at (1.35,0.76) {};
  \node[font=\tiny, above=1pt, text=blue!70!black] at (m1) {\texttt{m1}};
  \node[s2] (m2) at (1.35,0.12) {};
  \node[font=\tiny, above=1pt, text=red!70!black] at (m2) {\texttt{m2}};
  \node[s1] (c1) at (2.86,1.40) {};
  \node[font=\tiny, above=1pt, text=blue!70!black] at (c1) {\texttt{c1}};
  \node[s1] (c2) at (3.30,2.04) {};
  \node[font=\tiny, above=1pt, text=blue!70!black] at (c2) {\texttt{c2}};
  \node[s1] (c3) at (2.86,0.76) {};
  \node[font=\tiny, above=1pt, text=blue!70!black] at (c3) {\texttt{c3}};
\end{tikzpicture}

%% file: figures/small_network_depth.tex
\begin{tikzpicture}[v/.style={circle, fill, inner sep=0.5pt}]
  \tikzset{sm1/.style={circle, fill=gray, inner sep=1.2pt}}
  \tikzset{s0/.style={circle, fill=black, inner sep=1.2pt}}
  \tikzset{s1/.style={circle, fill=blue!70!black, inner sep=1.2pt}}
  \tikzset{s2/.style={circle, fill=red!70!black, inner sep=1.2pt}}
  \tikzset{s3/.style={circle, fill=violet!70!black, inner sep=1.2pt}}
  \tikzset{s4/.style={circle, fill=green!50!black, inner sep=1.2pt}}
  \draw[very thin] (2.43,2.04) -- (2.21,2.68);
  \draw[very thin] (2.43,2.04) -- (3.07,2.68);
  \draw[very thin] (1.56,2.04) -- (3.07,2.68);
  \draw[very thin] (1.56,2.04) -- (1.13,2.68);
  \draw[very thin] (0.25,1.40) -- (1.13,2.68);
  \draw[very thin] (0.25,1.40) -- (0.27,2.04);
  \draw[very thin] (0.70,1.40) -- (0.27,2.04);
  \draw[very thin] (0.70,1.40) -- (2.21,2.68);
  \draw[very thin] (1.77,1.40) -- (2.21,2.68);
  \draw[very thin] (1.77,1.40) -- (2.43,2.04);
  \draw[very thin] (1.77,1.40) -- (1.56,2.04);
  \draw[very thin] (0.70,0.76) -- (0.25,1.40);
  \draw[very thin] (0.70,0.76) -- (0.70,1.40);
  \draw[very thin] (1.35,0.76) -- (1.77,1.40);
  \draw[very thin] (1.35,0.76) -- (0.25,1.40);
  \draw[very thin] (1.35,0.12) -- (1.35,0.76);
  \draw[very thin] (1.35,0.12) -- (0.70,0.76);
  \draw[very thin] (1.35,0.12) -- (3.07,2.68);
  \draw[very thin] (2.86,1.40) -- (3.30,2.04);
  \draw[very thin] (2.86,1.40) -- (2.21,2.68);
  \draw[very thin] (3.30,2.04) -- (2.86,0.76);
  \draw[very thin] (3.30,2.04) -- (3.07,2.68);
  \draw[very thin] (2.86,0.76) -- (2.86,1.40);
  \draw[very thin] (2.86,0.76) -- (1.13,2.68);
  \node[s0] (a1) at (2.21,2.68) {};
  \node[font=\tiny, above=1pt, text=black] at (a1) {\texttt{a1}};
  \node[s0] (a2) at (3.07,2.68) {};
  \node[font=\tiny, above=1pt, text=black] at (a2) {\texttt{a2}};
  \node[s0] (a3) at (1.13,2.68) {};
  \node[font=\tiny, above=1pt, text=black] at (a3) {\texttt{a3}};
  \node[s0] (a4) at (0.27,2.04) {};
  \node[font=\tiny, above=1pt, text=black] at (a4) {\texttt{a4}};
  \node[s1] (e1) at (2.43,2.04) {};
  \node[font=\tiny, above=1pt, text=blue!70!black] at (e1) {\texttt{e1}};
  \node[s1] (e2) at (1.56,2.04) {};
  \node[font=\tiny, above=1pt, text=blue!70!black] at (e2) {\texttt{e2}};
  \node[s1] (e3) at (0.25,1.40) {};
  \node[font=\tiny, above=1pt, text=blue!70!black] at (e3) {\texttt{e3}};
  \node[s1] (e4) at (0.70,1.40) {};
  \node[font=\tiny, above=1pt, text=blue!70!black] at (e4) {\texttt{e4}};
  \node[s2] (f1) at (1.77,1.40) {};
  \node[font=\tiny, above=1pt, text=red!70!black] at (f1) {\texttt{f1}};
  \node[s2] (f2) at (0.70,0.76) {};
  \node[font=\tiny, above=1pt, text=red!70!black] at (f2) {\texttt{f2}};
  \node[s3] (m1) at (1.35,0.76) {};
  \node[font=\tiny, above=1pt, text=violet!70!black] at (m1) {\texttt{m1}};
  \node[s4] (m2) at (1.35,0.12) {};
  \node[font=\tiny, above=1pt, text=green!50!black] at (m2) {\texttt{m2}};
  \node[sm1] (c1) at (2.86,1.40) {};
  \node[font=\tiny, above=1pt, text=gray] at (c1) {\texttt{c1}};
  \node[sm1] (c2) at (3.30,2.04) {};
  \node[font=\tiny, above=1pt, text=gray] at (c2) {\texttt{c2}};
  \node[sm1] (c3) at (2.86,0.76) {};
  \node[font=\tiny, above=1pt, text=gray] at (c3) {\texttt{c3}};
\end{tikzpicture}

%% file: sections/plugins.tex
\section{The LeanNets Plugin}
\label{sec:plugins}

The core store is domain-agnostic; all domain logic lives in plugins. In this section, we present \textbf{LeanNets}, the first plugin, which turns Astrolabe into a bridge between informal and formal mathematics.

\subsection{Motivation}

Mathematical knowledge is rich in semantic relationships. A theorem may \textbf{unfold} a definition, \textbf{rewrite by} a lemma, \textbf{apply} a proposition to a subterm, or \textbf{specialize} a general result to a particular case. These distinctions matter: knowing \textbf{that} Theorem~A depends on Definition~B is less informative than knowing \textbf{how}. Does A unfold B, pattern-match on B's constructors, or merely pass B to another lemma? The \texttt{leanblueprint} tool~\cite{Massot2020} pioneered the use of visual dependency graphs for Lean formalization projects and has become essential infrastructure for large-scale efforts. LeanArchitect~\cite{Zhu2026} further automates blueprint generation. On the compiler side, Lean produces fine-grained declaration-level dependency graphs~\cite{LiPengSeverini2026}. These tools represent dependencies as directed edges labeled ``uses.'' LeanNets builds on this foundation by adding semantic content to each edge, recording \textbf{how} a dependency is used.

\begin{figure}[ht]
\centering
\begin{tikzpicture}[
  bp/.style={draw, rounded corners=2pt, fill=green!10, font=\scriptsize, inner sep=3pt, minimum width=1.4cm, minimum height=0.5cm},
  decl/.style={circle, fill, inner sep=1.5pt},
  arr/.style={->, >=Stealth, thin},
  darr/.style={->, >=Stealth, thin, red!60!black}
]
\node[font=\small\bfseries] at (1.5, 2.8) {Blueprint};
\node[bp] (la) at (1.5, 2) {\texttt{length\_append}};
\node[bp] (lm) at (0, 0.5) {\texttt{length\_map}};
\node[bp] (mm) at (3, 0.5) {\texttt{map\_map}};
\node[bp] (len) at (1.5, -0.8) {\texttt{List.length}};
\draw[arr] (la) -- node[left, font=\tiny] {uses} (lm);
\draw[arr] (la) -- node[right, font=\tiny] {uses} (len);
\draw[arr] (lm) -- node[left, font=\tiny] {uses} (len);
\draw[arr] (mm) -- node[right, font=\tiny] {uses} (len);

\node[font=\small\bfseries] at (7.5, 2.8) {Declaration graph};
\node[decl, label={[font=\scriptsize]above:\texttt{length\_append}}] (a) at (7.5, 2) {};
\node[decl, label={[font=\scriptsize]left:\texttt{length\_map}}] (b) at (5.8, 0.8) {};
\node[decl, label={[font=\scriptsize]right:\texttt{map\_map}}] (c) at (9.2, 0.8) {};
\node[decl, label={[font=\scriptsize]below:\texttt{List.length}}] (d) at (7.5, -0.3) {};
\node[decl, label={[font=\scriptsize]left:\texttt{List.rec}}] (rec) at (5.5, -0.5) {};
\node[decl, label={[font=\scriptsize]right:\texttt{Nat.add}}] (add) at (9.5, -0.5) {};
\node[decl, label={[font=\scriptsize]below:\texttt{List.map}}] (map) at (7.5, -1.3) {};
\draw[darr] (a) -- (b);
\draw[darr] (a) -- (d);
\draw[darr] (a) -- (add);
\draw[darr] (b) -- (d);
\draw[darr] (b) -- (map);
\draw[darr] (c) -- (map);
\draw[darr] (c) -- (rec);
\draw[darr] (d) -- (rec);
\draw[darr, dashed, gray] (a) to[bend right=15] (rec);
\draw[darr, dashed, gray] (a) to[bend left=15] (map);
\end{tikzpicture}
\caption{Left: a \texttt{leanblueprint} dependency graph. Each edge is labeled ``uses'' with no further content. Right: a declaration-level dependency graph extracted from Lean~4 compilation artifacts. Both representations lose the \textbf{semantic content} of each edge.}
\label{fig:blueprint-vs-decl}
\end{figure}
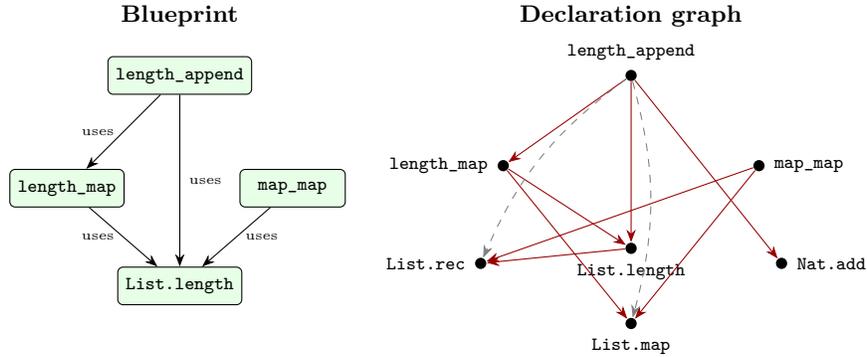

This gap has practical consequences for autoformalization. Large formalization projects are typically organized as blueprint dependency graphs, and autoformalization agents decompose their work along this structure. Harmonic's Aristotle~\cite{Aristotle2025} uses blueprint-structured decomposition to break problems into lemmas and conjectures; Axiom~\cite{Axiom2025} treats autoformalization as hierarchical planning with dependency retrieval; Math Inc.'s Gauss~\cite{Gauss2026} completed the Strong Prime Number Theorem formalization by working through a \texttt{leanblueprint} dependency graph node by node. In all three cases, the dependency annotations available to the agent are limited to ``uses,'' with no distinction between proof strategies.

Astrolabe's data model offers a finer granularity: every edge is an entry with a full \texttt{record}, so the nature of each dependency can be made explicit (unfolding, rewriting, specialization, application).
\subsection{Network Extraction}
The LeanNets plugin restricts attention to entries in $F_1$ with $\mathrm{w} \leq 1$, that is, atoms and width-1 entries whose refs are all atoms. This two-dimensional constraint (depth $\leq 1$, width $\leq 1$) yields a directed graph: atoms become \textbf{nodes}, and width-1 entries become \textbf{directed edges} ($\texttt{ref[0]} \to \texttt{ref[1]}$), with edge semantics derived from the record's \texttt{sort} field. Figure~\ref{fig:entry-view} shows the entry view; Figure~\ref{fig:network-view} shows the extracted network.

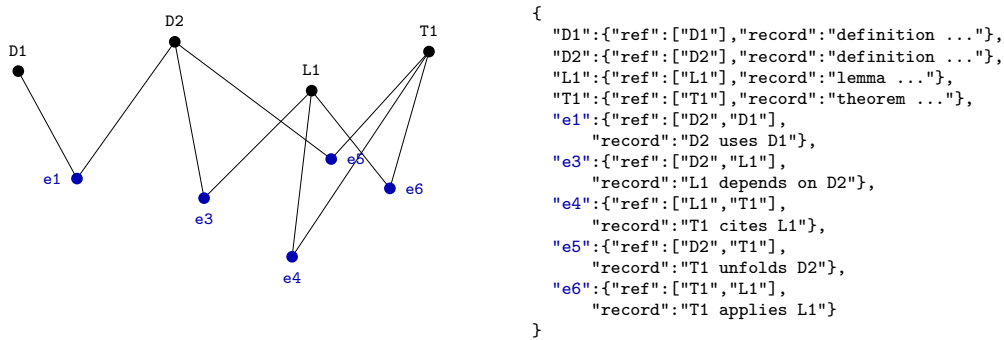
\begin{figure}[ht]
\centering
\begin{minipage}[t]{0.48\textwidth}
\vspace{0pt}
\centering
\scalebox{1.3}{%
\begin{tikzpicture}[v/.style={circle, fill, inner sep=0.5pt}]
  \tikzset{s0/.style={circle, fill=black, inner sep=1.2pt}}
  \tikzset{s1/.style={circle, fill=blue!70!black, inner sep=1.2pt}}
  \draw[very thin] (0.6,1.3) -- (0.0,2.4);
  \draw[very thin] (0.6,1.3) -- (1.6,2.7);
  \draw[very thin] (1.9,1.1) -- (1.6,2.7);
  \draw[very thin] (1.9,1.1) -- (3.0,2.2);
  \draw[very thin] (3.2,1.5) -- (1.6,2.7);
  \draw[very thin] (3.2,1.5) -- (4.2,2.6);
  \draw[very thin] (2.8,0.5) -- (3.0,2.2);
  \draw[very thin] (2.8,0.5) -- (4.2,2.6);
  \draw[very thin] (3.8,1.2) -- (4.2,2.6);
  \draw[very thin] (3.8,1.2) -- (3.0,2.2);
  \node[s0] (D1) at (0.0,2.4) {};
  \node[font=\tiny, above=1pt] at (D1) {\texttt{D1}};
  \node[s0] (D2) at (1.6,2.7) {};
  \node[font=\tiny, above=1pt] at (D2) {\texttt{D2}};
  \node[s0] (L1) at (3.0,2.2) {};
  \node[font=\tiny, above=1pt] at (L1) {\texttt{L1}};
  \node[s0] (T1) at (4.2,2.6) {};
  \node[font=\tiny, above=1pt] at (T1) {\texttt{T1}};
  \node[s1] (e1) at (0.6,1.3) {};
  \node[font=\tiny, left=1pt, text=blue!70!black] at (e1) {\texttt{e1}};
  \node[s1] (e3) at (1.9,1.1) {};
  \node[font=\tiny, below=1pt, text=blue!70!black] at (e3) {\texttt{e3}};
  \node[s1] (e4) at (2.8,0.5) {};
  \node[font=\tiny, below=1pt, text=blue!70!black] at (e4) {\texttt{e4}};
  \node[s1] (e5) at (3.2,1.5) {};
  \node[font=\tiny, right=1pt, text=blue!70!black] at (e5) {\texttt{e5}};
  \node[s1] (e6) at (3.8,1.2) {};
  \node[font=\tiny, right=1pt, text=blue!70!black] at (e6) {\texttt{e6}};
\end{tikzpicture}%
}
\end{minipage}%
\hspace{4pt}%
\begin{minipage}[t]{0.48\textwidth}
\vspace{0pt}
{\scriptsize\ttfamily
\begin{tabular}{@{}l@{}}
\{\\
\ \ \textcolor{black}{"D1"}:\{"ref":["\textcolor{black}{D1}"],"record":"definition ..."\},\\
\ \ \textcolor{black}{"D2"}:\{"ref":["\textcolor{black}{D2}"],"record":"definition ..."\},\\
\ \ \textcolor{black}{"L1"}:\{"ref":["\textcolor{black}{L1}"],"record":"lemma ..."\},\\
\ \ \textcolor{black}{"T1"}:\{"ref":["\textcolor{black}{T1}"],"record":"theorem ..."\},\\
\ \ \textcolor{blue!70!black}{"e1"}:\{"ref":["\textcolor{black}{D2}","\textcolor{black}{D1}"],\\
\ \ \ \ \ \ "record":"D2 uses D1"\},\\
\ \ \textcolor{blue!70!black}{"e3"}:\{"ref":["\textcolor{black}{D2}","\textcolor{black}{L1}"],\\
\ \ \ \ \ \ "record":"L1 depends on D2"\},\\
\ \ \textcolor{blue!70!black}{"e4"}:\{"ref":["\textcolor{black}{L1}","\textcolor{black}{T1}"],\\
\ \ \ \ \ \ "record":"T1 cites L1"\},\\
\ \ \textcolor{blue!70!black}{"e5"}:\{"ref":["\textcolor{black}{D2}","\textcolor{black}{T1}"],\\
\ \ \ \ \ \ "record":"T1 unfolds D2"\},\\
\ \ \textcolor{blue!70!black}{"e6"}:\{"ref":["\textcolor{black}{T1}","\textcolor{black}{L1}"],\\
\ \ \ \ \ \ "record":"T1 applies L1"\}\\
\}
\end{tabular}
}
\end{minipage}
\caption{Entry view: \textcolor{black}{black} = atoms (D1, D2, L1, T1), \textcolor{blue!70!black}{blue} = width-1 entries (e1--e5) carrying open-ended semantic records.}
\label{fig:entry-view}
\end{figure}

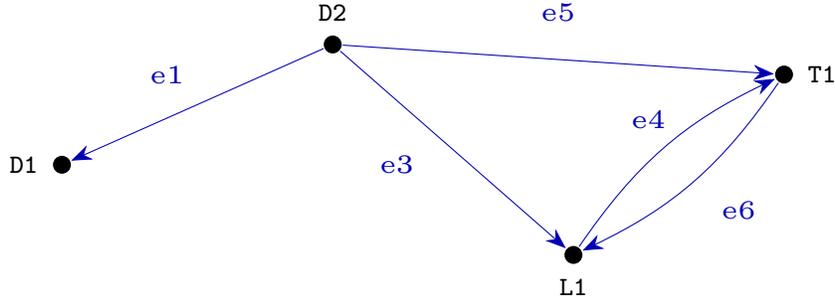
\begin{figure}[ht]
\centering
\scalebox{2.0}{%
\begin{tikzpicture}[v/.style={circle, fill, inner sep=0.5pt}]
  \tikzset{s0/.style={circle, fill=black, inner sep=1.2pt}}
  \node[s0] (D1) at (0.0,1.4) {};
  \node[font=\tiny, left=1pt] at (D1) {\texttt{D1}};
  \node[s0] (D2) at (1.8,2.2) {};
  \node[font=\tiny, above=1pt] at (D2) {\texttt{D2}};
  \node[s0] (L1) at (3.4,0.8) {};
  \node[font=\tiny, below=1pt] at (L1) {\texttt{L1}};
  \node[s0] (T1) at (4.8,2.0) {};
  \node[font=\tiny, right=1pt] at (T1) {\texttt{T1}};
  \draw[->, >=Stealth, very thin, blue!70!black] (D2) -- (D1);
  \node[font=\tiny, text=blue!70!black] at (0.7,2.0) {e1};
  \draw[->, >=Stealth, very thin, blue!70!black] (D2) -- (L1);
  \node[font=\tiny, text=blue!70!black, left=1pt] at (2.5,1.4) {e3};
  \draw[->, >=Stealth, very thin, blue!70!black] (D2) -- (T1);
  \node[font=\tiny, text=blue!70!black, above=1pt] at (3.3,2.2) {e5};
  \draw[->, >=Stealth, very thin, blue!70!black, bend left=15] (L1) to (T1);
  \node[font=\tiny, text=blue!70!black] at (3.9,1.7) {e4};
  \draw[->, >=Stealth, very thin, blue!70!black, bend left=15] (T1) to (L1);
  \node[font=\tiny, text=blue!70!black] at (4.5,1.1) {e6};
\end{tikzpicture}%
}
\caption{Network view of the same knowledge base. Atoms become nodes; width-1 entries become directed edges ($\texttt{ref[0]} \to \texttt{ref[1]}$). Edge labels (e1--e5) correspond to the entries in Figure~\ref{fig:entry-view}.}
\label{fig:network-view}
\end{figure}

\subsection{Record Conventions}

The LeanNets plugin interprets the \texttt{record} field as structured JSON with domain-specific fields. We illustrate with two example conventions, corresponding to informal and formal mathematics; the specific fields may evolve as the framework matures.

\paragraph{Informal mathematics: \texttt{source: "tex"}.}

For informal mathematical content, the \texttt{sort} field is derived from the \LaTeX\ environment (e.g., \verb|\begin{theorem}| $\to$ \texttt{sort: "theorem"}). The record carries the mathematical statement in natural language. The following illustrates the format used in the current prototype:

\medskip
{\scriptsize\ttfamily
\begin{tabular}{@{}l@{}}
\{"sort":"definition", "source":"tex", "title":"Compact Space",\\
\ \ "notes":"A subset $K$ is compact if every open cover has a finite subcover."\}\\[4pt]
\{"sort":"theorem", "source":"tex", "title":"Heine-Borel",\\
\ \ "notes":"A subset of $\mathbb{R}^n$ is compact iff it is closed and bounded."\}
\end{tabular}
}
\medskip

Atoms are definitions, theorems, lemmas, propositions, corollaries, and examples. Higher-width entries represent relationships between them: a proof references the theorem it proves and the lemmas it uses. The \texttt{notes} field supports \LaTeX\ math and inline references via \verb|\entryref{hash}{text}|.

\paragraph{Formal mathematics: \texttt{source: "lean"}.}

For formal mathematics, the \texttt{sort} field mirrors the Lean keyword (\texttt{def} $\to$ \texttt{definition}, \texttt{theorem} $\to$ \texttt{theorem}). The \texttt{content} field stores the source code, and \texttt{state} tracks formalization progress. An example from the current prototype:

\medskip
{\scriptsize\ttfamily
\begin{tabular}{@{}l@{}}
\{"sort":"theorem", "source":"lean", "title":"List.length\_append",\\
\ \ "state":"proven", "content":"theorem length\_append (l1 l2 : List a) : ..."\}\\[4pt]
\{"sort":"definition", "source":"lean", "title":"List.length",\\
\ \ "state":"checked", "content":"def length : List a -> Nat | [] => 0 | ..."\}
\end{tabular}
}
\medskip

Formal dependencies can be generated automatically by a source plugin that parses Lean's compilation artifacts. Building on the dependency structure provided by \texttt{leanblueprint}~\cite{Massot2020} and declaration graphs~\cite{LiPengSeverini2026}, Astrolabe edges carry open-ended content: a dependency can record ``rewrites by,'' ``unfolds definition of,'' or ``applies to subterm.''

\paragraph{Statement--proof separation.}
In both conventions, a theorem's \textbf{statement} and its \textbf{proof} are stored as separate atoms, connected by a width-1 edge. In Lean, a single \texttt{theorem} declaration contains both the type signature (statement) and the tactic body (proof); we split them into two entries during parsing. A statement and its proof have different lifecycles: the statement may be stable while the proof is still being revised. Separating them decouples their identities---changing a proof does not change the statement's hash, since each entry's identity depends only on its own record. This also means a single statement can admit multiple proofs (e.g., a constructive proof and a classical proof), represented as separate atoms linked to the same statement. The separation mirrors the structure of informal mathematics, where \verb|\begin{theorem}| and \verb|\begin{proof}| are distinct \LaTeX\ environments, making cross-source correspondence between \texttt{tex} and \texttt{lean} entries straightforward: the informal statement maps to the formal statement, and the informal proof sketch maps to the formal proof.

\noindent Example (\texttt{tex}):

\medskip
{\scriptsize\ttfamily
\begin{tabular}{@{}l@{}}
\textcolor{black}{"T"}:\{"ref":["\textcolor{black}{T}"],\\
\ \ "record":"\{"sort":"theorem","source":"tex","title":"Heine-Borel",\\
\ \ \ \ "notes":"A subset of $\mathbb{R}^n$ is compact iff closed and bounded."\}"\}\\[3pt]
\textcolor{black}{"P"}:\{"ref":["\textcolor{black}{P}"],\\
\ \ "record":"\{"sort":"proof","source":"tex",\\
\ \ \ \ "notes":"By contradiction, extract a sequence ..."\}"\}\\[3pt]
\textcolor{blue!70!black}{"e"}:\{"ref":["\textcolor{black}{T}","\textcolor{black}{P}"],"record":"statement-proof link"\}
\end{tabular}
}
\medskip

\noindent Example (\texttt{lean}):

\medskip
{\scriptsize\ttfamily
\begin{tabular}{@{}l@{}}
\textcolor{black}{"T'"}:\{"ref":["\textcolor{black}{T'}"],\\
\ \ "record":"\{"sort":"theorem","source":"lean","title":"IsCompact.isClosed",\\
\ \ \ \ "state":"proven","content":"theorem IsCompact.isClosed (h : IsCompact s) : IsClosed s"\}"\}\\[3pt]
\textcolor{black}{"P'"}:\{"ref":["\textcolor{black}{P'}"],\\
\ \ "record":"\{"sort":"proof","source":"lean",\\
\ \ \ \ "content":"by intro x hx; exact h.isSeqCompact.isClosed hx"\}"\}\\[3pt]
\textcolor{blue!70!black}{"e'"}:\{"ref":["\textcolor{black}{T'}","\textcolor{black}{P'}"],"record":"statement-proof link"\}
\end{tabular}
}
\medskip

\subsection{Cross-Source Edges}

Edges (width-1 entries) connect two atoms. Their record fields fall into two categories:
\begin{itemize}
  \item \textbf{Structural fields} (\texttt{sort}, \texttt{source}) are \textbf{inherited}: they are fully determined by the corresponding fields of the referenced atoms. An edge connecting a \texttt{theorem} to a \texttt{definition} inherits the sort pair \texttt{(theorem, definition)}; an edge between a \texttt{tex} atom and a \texttt{lean} atom inherits the source pair \texttt{(tex, lean)}. No manual annotation is needed for these fields. This inheritance extends naturally to higher-width entries, whose structural fields are determined by the tuple of their refs' fields.
  \item \textbf{Semantic fields} (\texttt{notes}, \texttt{title}) are \textbf{authored}: they describe the nature of the relationship (e.g., ``unfolds definition,'' ``rewrites by lemma'') and are the edge's own contribution. This is the layer of information that Astrolabe adds on top of the dependency structure that blueprint tools provide.
\end{itemize}

\noindent The sort pair is auto-derived from the pair of source sorts:

\begin{table}[H]
\centering
\begin{tabular}{ll}
\toprule
Edge sort pair & Meaning \\
\midrule
\texttt{(theorem, proof)} & Statement--proof link \\
\texttt{(theorem, definition)} & Depends on definition \\
\texttt{(proof, lemma)} & Proof cites lemma \\
\texttt{(theorem, theorem)} & Cross-source correspondence \\
\bottomrule
\end{tabular}
\end{table}

When both \texttt{tex} and \texttt{lean} entries coexist in the same file, a \texttt{(theorem, theorem)} edge with one \texttt{tex} source and one \texttt{lean} source marks the formalization correspondence. This makes the bridge between informal and formal mathematics explicit and navigable in the network view.

\subsection{Semantic Propagation}
Since identity depends only on the record (\S\ref{sec:data-model}), modifying an atom does not automatically affect other entries' hashes. But semantic dependency is not captured by identity alone. Consider: Definition~$D$ changes; Edge~$E$ (``Theorem~$T$ depends on~$D$'') still references~$D$, but Theorem~$T$ itself is unaffected at the hash level, since its ref is $[\text{self}]$ and its record may not mention~$D$. To capture this, the LeanNets plugin adds a propagation layer on the skeleton graph: when an atom changes, the directed edge structure is traversed in reverse to identify all atoms that \textbf{semantically} depend on the changed atom. If Definition~$D$ is modified, every atom reachable by following edges backward from~$D$ is flagged as affected. This is a straightforward reverse BFS on the skeleton graph. Figure~\ref{fig:semantic-propagation} illustrates the difference.

\begin{figure}[ht]
\centering
\begin{tikzpicture}[
  atom/.style={circle, draw, minimum size=18pt, inner sep=0pt, font=\scriptsize},
  edge/.style={rectangle, draw, rounded corners=2pt, minimum size=14pt, inner sep=2pt, font=\scriptsize},
  changed/.style={fill=red!25, draw=red!70!black},
  rehashed/.style={fill=orange!20, draw=orange!70!black},
  affected/.style={fill=red!15, draw=red!60!black, dashed},
  unaffected/.style={fill=gray!10, draw=gray!50},
  arr/.style={->, >=Stealth, thin},
]
\node[font=\small\bfseries] at (1.5, 2.8) {Identity only};
\node[atom, changed] (D1) at (0, 1.5) {$D$};
\node[atom, unaffected] (T1) at (3, 1.5) {$T$};
\node[edge, unaffected] (E1) at (1.5, 0.3) {$E$};
\draw[arr] (E1) -- (D1);
\draw[arr] (E1) -- (T1);
\node[font=\tiny, red!70!black] at (0, 0.8) {changed};
\node[font=\tiny, gray!70] at (1.5, -0.2) {unaffected};
\node[font=\tiny, gray!70] at (3, 0.8) {unaffected};

\node[font=\small\bfseries] at (7.5, 2.8) {Semantic propagation};
\node[atom, changed] (D2) at (6, 1.5) {$D$};
\node[atom, affected] (T2) at (9, 1.5) {$T$};
\draw[arr, thick, red!50!black] (D2) -- node[above, font=\tiny] {depends on} (T2);
\node[font=\tiny, red!70!black] at (6, 0.8) {changed};
\node[font=\tiny, red!60!black] at (9, 0.8) {flagged};
\end{tikzpicture}
\caption{Left: since identity depends only on the record, modifying~$D$ does not affect Edge~$E$ or Theorem~$T$ at the hash level. Right: semantic propagation on the skeleton graph traverses the dependency edge in reverse, flagging~$T$ as semantically affected by the change in~$D$.}
\label{fig:semantic-propagation}
\end{figure}
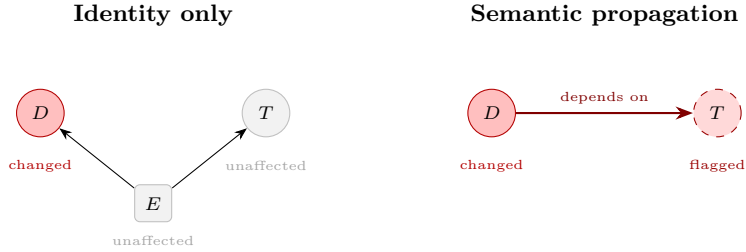

\subsection{Visualization and Interaction}
\label{sec:visualization}

The network view renders the extracted graph on a 2D canvas using a force-directed layout. Node \textbf{size} encodes a selectable graph metric (degree, PageRank, betweenness centrality, Katz centrality, HITS hub/authority scores, DAG depth, or reachability count). Node \textbf{color} encodes either a categorical attribute (sort, community via greedy modularity, spectral clustering) or a continuous metric (PageRank, betweenness, DAG depth). Nodes can be grouped by \textbf{clustering} (Louvain, sort, depth, or spectral), with a tightness parameter controlling the intra-cluster attractive force. All metrics are computed independently per source to prevent cross-source edges from dominating centrality scores.

%% file: sections/conclusion.tex
\section{Future Work}
\label{sec:conclusion}

We outline three directions for future work.

\begin{itemize}
  \item \textbf{Mathematical foundations.} Barkeshli, Douglas, and Freedman~\cite{BarkeshliDouglasFreedman2026} define a \emph{universal proof hypergraph} whose directed $(p,q)$-hyperedges represent derivation rules with $p$ input propositions and $q$ output propositions. Astrolabe's nerves are content-oriented: the reference list is ordered but does not distinguish inputs from outputs. Characterizing the precise embedding and reduction between the two representations is a natural theoretical question.

  \item \textbf{Content-addressing formal mathematics.} Applying content-addressing to type expressions in the Calculus of Inductive Constructions raises questions about normalization strategy (structural vs.\ reducible-normalization), the choice between Merkle DAG and name-based hashing, and the theoretical limitations imposed by the absence of unique normal forms in dependent type theory. These questions are explored in~\cite{AstroCA}.

  \item \textbf{Network analysis for autoformalization.} Whether network-analytic signals (centrality, community structure, DAG depth) extracted from the Astrolabe store can improve premise retrieval and proof strategy selection for autoformalization agents is an empirical question. Preliminary evidence from the Mathlib dependency network~\cite{LiPengSeverini2026} suggests that structural features alone carry predictive power for premise selection.
\end{itemize}